\documentclass[12pt,reqno]{amsart}
\usepackage{latexsym, amsmath, amscd, amssymb, amsthm,longtable} 
\usepackage[T2A]{fontenc}
 \usepackage[cp1251]{inputenc}
\usepackage[english]{babel}
\newtheorem{theorem}{Theorem}

\usepackage{mathtools}

\newcommand{\Map}{\operatorname{Map}}

\newcommand{\bs}{\boldsymbol}
\usepackage{cite}

\begin{document}

\noindent

\title { Relative Invariants from Moving Frames on an Extended Manifold}

\author{Leonid Bedratyuk} 
\address{ Khmelnytsky National University, Ukraine}
\email{leonidbedratyuk@khmnu.edu.ua}

\begin{abstract}

A constructive modification of the moving frame method is developed in this paper for the construction of relative invariants of regular Lie group actions. Let a relative invariant $I$ of weight $\omega$ transform according to the rule  
$$  
I(g \cdot \boldsymbol x) = \mu(g, \boldsymbol x)^{\omega} I(\boldsymbol x),  
$$  
where $\mu: G \times \mathcal{M} \to \mathbb{R}^\times$ is a scalar multiplier (1-cocycle). It is shown that the cocycle property of $\mu$ is equivalent to the well-definedness of the twisted group action on the extended manifold $\widehat{\mathcal{M}} = \mathcal{M} \times \mathbb{R}^\times$, and that relative invariants on $\mathcal{M}$ are in one-to-one correspondence with absolute invariants of this action on $\widehat{\mathcal{M}}$.

The main result is that, given a moving frame, the invariantization of the multiplier is a canonical relative invariant of weight $-1$. This enables the constructive realization of any weight and yields an explicit formula for an arbitrary relative invariant in terms of the fundamental absolute invariants and the invariantized multiplier. Examples are provided to demonstrate the application of the proposed approach for the projective group $PGL(3, \mathbb{R})$.


\end{abstract}

\maketitle

\section{Introduction}

Let $G$ be an $r$-dimensional Lie group acting smoothly on an $n$-dimensional manifold $\mathcal M$.
This action induces an action of $G$ on the field $\mathcal F(\mathcal M)$ of rational functions on $\mathcal M$.
A function $f\in\mathcal F(\mathcal M)$ is called a \emph{relative invariant} with a smooth multiplier
$\mu:G\times\mathcal M\to\mathbb R^\times$ if
\begin{equation}\label{m_oz}
f(g\cdot \boldsymbol x)=\mu(g,\boldsymbol x)\,f(\boldsymbol x),\qquad \forall\, g\in G,\ \boldsymbol x\in\mathcal M.
\end{equation}
More generally, one can introduce a rational \emph{weight} $\omega\in\mathbb Q$ and replace $\mu$ by the weighted
multiplier, i.e.
\[
f(g\cdot\boldsymbol x)=\mu(g,\boldsymbol x)^{\omega}\,f(\boldsymbol x).
\]
In this terminology, a relative invariant of weight $0$ is an \emph{absolute} invariant.

The classification problem for relative invariants of transformation group actions arises throughout modern
mathematics and its applications. Relative invariants occur naturally in the setting of affine and projective
transformations, where, for instance, area or volume is no longer preserved but is scaled by the determinant of
the transformation matrix. A general description of relative invariants makes it possible to treat a broad range
of invariant structures as sections of vector bundles preserved by the induced group action.
Relative invariants play a fundamental role in differential geometry (invariant vector fields, differential forms,
volume elements) \cite{Graustein, Olver1997, Olver2007, Bott}, in the theory of invariant differential operators
\cite{Helgason}, in the study of invariant metrics and connections \cite{KobayashiNomizu}, in number theory
(automorphic and modular forms, Epstein zeta functions) \cite{Langlands, Miyake}, in classical mechanics
(relative integral invariants of Poincar\'e--Cartan type and symplectic geometry) \cite{AbrahamMarsden, Arnold},
and in applied problems such as computer vision \cite{Olver2023, Wang, Open}, the theory of special functions
(in particular, Mellin--Barnes type integral representations, matrix-argument functions, and Whittaker functions)
\cite{Miller, VilenkinKlimyk}, and quantum mechanics (gauge invariance and the transformation properties of wave
functions viewed as relative invariants) \cite{Weyl}.

Despite their wide applicability, the systematic theory of relative invariants is far less developed than that of
absolute invariants. While the classification of absolute invariants for regular Lie group actions is a direct
consequence of the Frobenius theorem, an analogous general theory for relative invariants has only been achieved
more recently. In particular, \cite{Olver1997, Olver1995} completely solves the general classification problem
for relative invariants of regular multiplier representations of connected Lie group actions, providing an
{explicit formula} for the precise number of functionally independent relative invariants and establishing
their structure. The breadth of applicability is illustrated through a detailed treatment of several important
geometric constructions arising as relative invariants of (vector and, in particular, tensor) bundle actions,
including invariant vector fields and differential forms, purely geometric conditions for the existence of such
fields, and an extensive discussion of applications to invariant connections (including inhomogeneous multiplier
representations).

In practically important problems, however, a central issue remains the {constructive} extraction of
generating relative invariants within a function class suitable for computation; the present work focuses on this
algorithmic aspect. Our goal is to develop a general methodology for constructing relative invariants using the
moving frame method---an effective constructive technique originating in the geometric ideas of Cartan and
developed in modern form by Fels and Olver \cite{Olver1998, Olver1999-1}. The moving frame method has become a
standard tool for producing absolute and differential invariants of regular Lie group actions. At the same time,
its potential for a systematic treatment of {relative} invariants has not been fully exploited in the
literature.

The proposed approach rests on two key ideas that reduce the problem to an explicit constructive algorithm.
First, we pass to the {extended space} $\widehat{\mathcal M} = \mathcal M\times\mathbb R^\times$ and introduce a {twisted
action} of $G$ in which the multiplier $\mu$ appears as the factor governing the transformation of the additional
coordinate. This reduces the construction of relative invariants on $\mathcal M$ to the construction of absolute
invariants for the extended action on $\widehat{\mathcal  M}$. Second, assuming that a moving frame for the action on $M$ has
already been constructed, we show that {invariantization} of the multiplier yields a canonical relative
invariant of weight $-1$. In this way, relative invariants emerge as a natural extension of absolute invariants
within the unified geometric paradigm of the moving frame method.

Our results also generalize and systematize observations that have appeared in concrete examples, in particular
for projective groups, where invariantization of the multiplier agrees with natural frame-dependent quantities
arising in projectively invariant constructions \cite{B-2025, B-2025-1}. Nevertheless, the main emphasis of this
paper is not on classification or global issues, but on the constructive procedure: given a moving frame, how to
efficiently produce relative invariants.

We also obtain a complete description of admissible multipliers via a cohomological interpretation. A multiplier
is naturally viewed as a $1$-cocycle in the multiplicative bar complex of the group action with coefficients in
$\mathcal F(M)^\times$, and the cocycle condition is equivalent to the functional cocycle identity for
multipliers. Using the triviality of the first cohomology group in this setting, we show that every $1$-cocycle is
a $1$-coboundary. Thus, the cohomological approach not only clarifies the geometric meaning of the multiplier as a
``change of gauge'', but also yields a closed-form classification of all multipliers.

\medskip
In a broader geometric context, one encounters the nontrivial {realizability} problem: whether a prescribed
$1$-cocycle can occur as the multiplier of some relative invariant, notably in connection with invariant divisors
and equivariant line bundles; see \cite{Krug}. For regular smooth actions in the class $C^\infty(M)$, local results
were obtained by Fels and Olver \cite{Olver1997}, whereas in the general case no universal criterion is known.
In the present work we remain in a constructive regular setting and in the class of rational functions, where the
existence of a moving frame implies the triviality of the first cohomology and hence the realizability of the
$1$-cocycle.

The paper is organized as follows. Section~1 provides a brief review of the moving frame method for regular Lie
group actions and introduces multipliers; in particular, we present the cohomological interpretation of
multipliers as $1$-cocycles and prove that, in the constructive class considered here, all such cocycles are
coboundaries. In Section~2 we propose a modification of the moving frame method for relative invariants: we
construct the extension $\widehat{\mathcal M}=\mathcal M\times\mathbb R^\times$ with a twisted action, establish
the correspondence between relative invariants on $\mathcal M$ and absolute invariants of the extended action, and
show that invariantization of the multiplier by the moving frame produces a canonical weight $-1$ relative
invariant and leads to an explicit normal form for an arbitrary relative invariant. Section~3 applies the general
results to the diagonal projective action of $PGL(3,\mathbb R)$ on configuration spaces: we construct a moving
frame, derive fundamental joint absolute invariants and an explicit invariantized Jacobian, and then construct a
new family of joint integral projective invariants of images that are invariant under projective deformations of
the domain.

In this paper, with a mild abuse of notation, we identify a point of a manifold with its local coordinates in a
chosen chart defined in a neighborhood of that point.

\section{The Moving Frame Method and Multipliers}

The modern form of the moving frame method was developed in \cite{Olver1998}, \cite{Olver1999-1}, although its origins can be traced back to works from the late 19th century. Below, we briefly present the essence of the method and describe multipliers in cohomological terms.

\subsection{The Moving Frame Method}

Let a Lie group $G$ act smoothly on a manifold $\mathcal M.$ A \textit{(right) moving frame} is a smooth $G$-equivariant map $\rho: \mathcal M \to G$ such that
$$
\rho(g \cdot \boldsymbol x) = \rho(\boldsymbol x) g^{-1},
$$
for all $g \in G$ and all $\boldsymbol x \in \mathcal M$.

The existence conditions for a moving frame are given in the following theorem:

\begin{theorem}[\cite{Olver1999-1}]\label{mt}
If $G$ acts on $\mathcal M$, then a moving frame exists in a neighborhood of
a point $\boldsymbol x \in \mathcal M$ if and only if $G$ acts freely and regularly near $\boldsymbol x$.
\end{theorem}

The regularity condition means that every point of $\mathcal M$ has arbitrarily small neighborhoods whose intersection with each orbit is a connected subset thereof. In particular, regularity implies that all group orbits are submanifolds of $\mathcal M$ of the same dimension, which coincides with the dimension of the group $G$. The action of $G$ is said to be free if for each point $\boldsymbol x \in \mathcal M$, the stabilizer  
$$
G_{\boldsymbol x} = \{g \in G \mid g \cdot \boldsymbol x = \boldsymbol x \}
$$ 
is trivial.

The action of $G$ on $\mathcal M$ induces an action on the field of rational functions $\mathcal{F}(\mathcal M)$ by the rule
$$
g \cdot F(\boldsymbol x) = F(g^{-1} \boldsymbol x).
$$

This yields the field of (absolute) invariants:
$$
\mathcal{F}(\mathcal M)^G = \{ F \in \mathcal{F}(\mathcal M) \mid g \cdot F = F,\quad \forall g \in G \}.
$$

The existence of a moving frame often allows for the construction of a minimal generating set of invariants for the field $\mathcal{F}(\mathcal M)^G$, using the fact that the moving frame maps all points on an orbit to the same point. Indeed, let $\boldsymbol x, \boldsymbol y \in \mathcal M$ lie on the same orbit, so that $\boldsymbol y = g \cdot \boldsymbol x$ for some $g \in G$. Then, by equivariance:
$$
\rho(\boldsymbol y) \cdot \boldsymbol y
= \rho(g \cdot \boldsymbol x) \cdot (g \cdot \boldsymbol x)
= \big(\rho(\boldsymbol x)\, g^{-1} \big) \cdot (g \cdot \boldsymbol x)
= \rho(\boldsymbol x) \cdot \boldsymbol x.
$$

This fixed point $\boldsymbol{x}^* = \rho(\boldsymbol x) \cdot \boldsymbol x$ is called the \textit{normalization} of the orbit point $\boldsymbol x$. The local coordinates of the normalized point
$$
\boldsymbol{x}^* = (I_1(\boldsymbol x), I_2(\boldsymbol x), \ldots, I_m(\boldsymbol x)), \quad m = \dim \mathcal M,
$$
do not depend on the choice of the orbit representative and form a complete system of invariants, from which $m - r$ invariants constitute a minimal generating set of the invariant field $\mathcal{F}(\mathcal M)^G$ \cite[Theorem~4.5]{Olver1999-1}.

\medskip 
For the practical computation of invariants, one introduces the operation of
\textit{invariantization} of functions.
The invariantization of a function $F \in \mathcal F(\mathcal M)$ with respect to a moving frame $\rho$
is the function $\iota(F) \in \mathcal F(\mathcal M)$ defined by
$$
\iota(F)(\boldsymbol x) = F(\rho(\boldsymbol x) \cdot \boldsymbol x).
$$

It is easy to see that $\iota^2 = \iota$, so that we obtain a projector
$$
\iota : \mathcal F(\mathcal M) \to \mathcal F(\mathcal M)^G.
$$
Thus, invariantization realizes a canonical projection of an arbitrary function onto its
$G$-invariant normalized form. In this sense, the invariantization operator is an analogue
of the Reynolds operator for finite groups: instead of discrete averaging over the group,
it uses geometric normalization.

\medskip

To construct a moving frame, one must determine the group parameters defining the element
$\rho(\boldsymbol x)$ as functions of the local coordinates of the point $\boldsymbol x$.
To this end, one fixes a cross-section $\mathcal K$ by $r$ equations
$$
\mathcal{K} = \{\, x_1 = c_1,\ \dots,\ x_r = c_r \,\} \subset \mathcal M,
$$
where $r = \dim G$ and $c_1, \dots, c_r$ are suitably chosen constants. These equations fix
$r$ coordinates of the normalized point.
The associated right moving frame $\rho(\boldsymbol x) \in G$ is obtained by solving
the \textit{normalization equations}
$$
\big(g \cdot \boldsymbol x\big)_j = c_j, \qquad j = 1, \dots, r.
$$

This is a system of $r$ equations, typically polynomial, in $r$ unknowns—the group parameters.
Theorem~\ref{mt} guarantees that this system admits a unique local solution, which yields the
moving frame $\rho(\boldsymbol x)$. Consequently, for an arbitrary point $\boldsymbol x$,
the coordinates of the normalized point, i.e., the invariantizations of the coordinate functions,
take the form
$$
\boldsymbol{x}^* = \big(c_1, \dots, c_r,\ I_1(\boldsymbol{x}), \dots, I_{m - r}(\boldsymbol{x})\big).
$$

Hence, the $m-r$ invariants $I_1(\boldsymbol{x}), \ldots, I_{m-r}(\boldsymbol{x})$ form a minimal
generating set of the invariant field: they are functionally independent, and any invariant
in $\mathcal{F}(\mathcal M)^G$ is a function of
$I_1(\boldsymbol{x}), \ldots, I_{m-r}(\boldsymbol{x})$.

We note that although the moving frame is defined locally, the invariants obtained by
invariantization are global, since they depend only on the orbit and not on the choice of a
representative point on the orbit.

\medskip

Finally, we remark that in most practically relevant cases the normalization equations are
linear with respect to the group parameters. In particular, this situation is characteristic
of the special projective group $PSL(n,\mathbb{R})$, the affine group $A(n,\mathbb{R})$, and its
subgroups, as shown in \cite{Olver2001}. This allows one to obtain explicit solutions of the
normalization system, in which the moving frame parameters are expressed as rational functions
of the coordinates $\boldsymbol{x}$. Consequently, all objects constructed on this basis—from
absolute invariants to invariantized multipliers—remain within the class of rational functions
on $\mathcal{M}$.

On the other hand, for problems such as the construction of classical invariants of binary forms,
the normalization equations are high-degree polynomial equations, which significantly complicates
the explicit determination of solutions even for binary quadratics; see \cite{Olver1999-2}.
In view of this, in the present work we restrict our attention to the practical case where
$\mathcal{F}(\mathcal{M})$ is the field of rational functions.

\subsection{Multipliers}

Which functions $\mu : G \times \mathcal M \to \mathbb{R}$ can serve as multipliers?
From the identity $g \cdot (h \cdot F) = (gh) \cdot F$, expanding both sides and equating them,
we obtain that multipliers must satisfy the identity
\begin{equation}\label{mult}
\mu(gh, \boldsymbol x) = \mu(g, h \boldsymbol x)\,\mu(h, \boldsymbol x), 
\qquad \mu(e,\boldsymbol x)=1.
\end{equation}

In the classical framework \cite{Olver1997}, multipliers arise naturally from a so-called
{change of gauge} on the space of functions.
Namely, fix an arbitrary nowhere-vanishing function (a gauge factor)
$\eta : \mathcal M \to \mathbb{R}^\times$ and modify the standard action on functions
by multiplication with $\eta$, that is, consider the transformation $F \mapsto \eta F$.
Then the transformed action acquires the factor
$$
\frac{\eta(g\!\cdot\!\boldsymbol x)}{\eta(\boldsymbol x)},
$$
which is easily verified to be a multiplier.
In this subsection, using methods of homological algebra, we show that this construction
in fact produces {all} possible multipliers.

\medskip

We recall the necessary notions.
Let
$$
\mathcal F(\mathcal M)^\times
=
\{\,f\in\mathcal F(\mathcal M)\mid f(\boldsymbol x)\neq 0\ \ \forall \boldsymbol x\in\mathcal M\,\}
$$
be the multiplicative abelian group of nowhere-vanishing functions on $\mathcal M$.
Consider the standard inhomogeneous \emph{multiplicative} bar complex associated with the group action:
$$
C^n = \Map(G^n \times \mathcal M,\mathbb{R}^\times),\qquad
(C^\bullet,d^\bullet):\quad
C^0 \xrightarrow{\,d^0\,} C^1 \xrightarrow{\,d^1\,} C^2 \xrightarrow{\,d^2\,}\cdots,
$$
where $C^0=\mathcal F(\mathcal M)^\times$.
The coboundary operators are given by the formulas (see \cite{Braun}):
$$
(d^0 c)(g,\boldsymbol x)
=
\frac{c(g\!\cdot\!\boldsymbol x)}{c(\boldsymbol x)},
\qquad c\in C^0,
$$
$$
(d^1 c)(g_1,g_2,\boldsymbol x)
=
\frac{c(g_1,g_2\!\cdot\!\boldsymbol x)\,c(g_2,\boldsymbol x)}
{c(g_1 g_2,\boldsymbol x)},
\qquad c\in C^1,
$$
and for $n\ge 1$ and $c\in C^n$ we have the standard expression
$$
\begin{aligned}
(d^n c)(g_1,\dots,g_{n+1};\boldsymbol x)
&=
c(g_2,\dots,g_{n+1};\boldsymbol x)
\prod_{i=1}^{n}
c(g_1,\dots,g_i g_{i+1},\dots,g_{n+1};\boldsymbol x)^{(-1)^i} \\
&\hspace{2.2cm}\cdot
c(g_1,\dots,g_n;\,g_{n+1}\!\cdot\!\boldsymbol x)^{(-1)^{n+1}}.
\end{aligned}
$$
These operators satisfy the identity $d^{n+1}\circ d^n = 1$
(the identity element, corresponding to $0$ in additive notation),
so the cocycle and coboundary groups are well defined:
$$
Z^n(G,\mathcal F(\mathcal M)^\times)=\ker d^n,
\qquad
B^n(G,\mathcal F(\mathcal M)^\times)=\operatorname{Im} d^{n-1},
\qquad
H^n = Z^n/B^n.
$$

In particular, the $1$--cocycle condition $(d^1 c)(g_1,g_2,\boldsymbol x)=1$,
written explicitly as
$$
\frac{c(g_1,g_2\!\cdot\!\boldsymbol x)\,c(g_2,\boldsymbol x)}
{c(g_1 g_2,\boldsymbol x)}=1,
$$
coincides with the multiplier identity \eqref{mult}.

\medskip

The following theorem establishes the classification of all multipliers by exploiting the triviality of the higher cohomology groups $H^n(G, \mathcal{F}(\mathcal{M})^\times)$ for $n \geq 1$, which is a direct consequence of the existence of a moving frame \cite[Section~5]{B-2025}.

\begin{theorem} Suppose that the action of $G$ on $\mathcal{M}$ is free and regular in a neighborhood of a point $\boldsymbol{x} \in \mathcal{M}$, and there exists a right moving frame $\rho: \mathcal{M} \to G$ defined in this neighborhood.
The following statements hold:
\begin{enumerate}
\item[(i)]
For any $f\in\mathcal F(\mathcal M)^\times$, the function
$$
\mu(g,\boldsymbol x)
=
\frac{f(g\!\cdot\!\boldsymbol x)}{f(\boldsymbol x)}
$$
is a multiplier, i.e., a $1$--cocycle: $d^1\mu = 1$.

\item[(ii)]
For every multiplier $\mu$ there exists $f\in\mathcal F(\mathcal M)^\times$ such that
$$
\mu(g,\boldsymbol x)
=
\frac{f(g\!\cdot\!\boldsymbol x)}{f(\boldsymbol x)}.
$$
\end{enumerate}
\end{theorem}

\begin{proof}
\textit{(i)}
Take $\mu=d^0 f$. Then, by definition of $d^0$ and $d^1$,
$$
(d^1\mu)(g_1,g_2,\boldsymbol x)
=
\frac{\mu(g_1,g_2\!\cdot\!\boldsymbol x)\,\mu(g_2,\boldsymbol x)}
{\mu(g_1 g_2,\boldsymbol x)}
=
\frac{
\displaystyle
\frac{f(g_1 g_2\!\cdot\!\boldsymbol x)}{f(g_2\!\cdot\!\boldsymbol x)}
\cdot
\frac{f(g_2\!\cdot\!\boldsymbol x)}{f(\boldsymbol x)}
}{
\displaystyle
\frac{f(g_1 g_2\!\cdot\!\boldsymbol x)}{f(\boldsymbol x)}
}
=1.
$$
Hence $\mu$ is a $1$--cocycle, i.e., a multiplier.

\smallskip
\textit{(ii)}
Let $\mu$ be an arbitrary multiplier. This is equivalent to the condition
$d^1\mu = 1$, i.e., $\mu\in \ker d^1 = Z^1$.
Since the first cohomology group is trivial, $H^1(C^\bullet)=1$,
we have $Z^1 = B^1 = \operatorname{Im} d^0$,
so every $1$--cocycle is a $1$--coboundary.
Therefore there exists $f\in C^0=\mathcal F(\mathcal M)^\times$ such that $\mu=d^0 f$,
which is precisely the formula
$$
\mu(g,\boldsymbol x)
=
\frac{f(g\!\cdot\!\boldsymbol x)}{f(\boldsymbol x)}.
$$
\end{proof}

\medskip

Thus, we have provided a complete characterization of the solutions to \eqref{mult} in the class of rational functions:
all of its solutions are parametrized by nowhere-vanishing functions on the manifold
$\mathcal M$.



\section{Modification of the Moving Frame Method for Relative Invariants}

The moving frame method in its standard form describes the field of absolute invariants for a given group action on a manifold. To systematically incorporate relative invariants with a fixed multiplier $\mu(g,\boldsymbol x)$, we extend the base manifold by adding one additional coordinate, and use $\mu$ to define a well-defined twisted group action on the extended manifold. On this extended space, relative invariants are naturally interpreted as absolute invariants, which allows us to directly apply the standard procedure of choosing a cross-section, constructing a moving frame, and performing invariantization. In addition, we obtain an explicit weight generator that complements the fundamental absolute invariants of the base manifold.

\subsection{Why Can't We Stay Within $\mathcal M$?}

Can one force the moving frame method, which by its nature constructs absolute invariants, to automatically generate relative invariants while remaining within $\mathcal M$?
There might be a temptation to define a new group action on functions in such a way that the absolute invariants of this new action coincide with the relative invariants under the standard action \eqref{m_oz}. This can indeed be done, for example, by defining a twisted action $\circ$ as follows:
$$
g \circ F(\boldsymbol x) = \mu(g^{-1}, \boldsymbol x)^{-1} F(g^{-1} \boldsymbol x), 
\qquad g \in G,\quad F \in \mathcal F(\mathcal M).
$$
Then the condition of absolute invariance $g \circ F = F$ is equivalent to
$$
\mu(g^{-1}, \boldsymbol x)^{-1} F(g^{-1} \boldsymbol x) = F(\boldsymbol x),
$$
or, after replacing $g \mapsto g^{-1}$,
$$
F(g \boldsymbol x) = \mu(g, \boldsymbol x)\, F(\boldsymbol x),
\qquad \forall\, g \in G,\ \boldsymbol x \in \mathcal M.
$$
Thus, an absolute invariant $F$ with respect to the $\circ$-action is exactly a relative invariant with multiplier $\mu$.

However, this $\circ$-action is defined only on the function space $\mathcal F(\mathcal M)$ and cannot be descended to an action on the manifold $\mathcal M$ itself. The reason is simple: any group action on $\mathcal M$ induces an automorphism of the function field $\mathcal F(\mathcal M)$ that preserves the multiplicative identity and products. But for the $\circ$-action, we have
$$
(g \circ 1)(\boldsymbol x) = \mu(g^{-1}, \boldsymbol x)^{-1} \neq 1,
\quad\text{and in general,}\quad
g \circ (FG) \neq (g \circ F)(g \circ G).
$$
Hence, such a twisting of the action only at the level of functions does not define a geometric group action on $\mathcal M$, and therefore the moving frame method cannot be directly applied on $\mathcal M$ to compute relative invariants in this way.

To nonetheless make use of the moving frame method, in the next subsection we introduce a correct geometric lifting of this action to an extended manifold by adding one additional “weight” coordinate, on which the multiplier $\mu$ is realized as a standard group action.

\subsection{Moving Frame on the Extended Manifold $\mathcal M$}

A fully geometric way to handle multiplier representations is to pass from the base manifold
\(\mathcal M\) to an {extended} manifold in which the multiplier becomes part of an ordinary group action.
This idea is already present in \cite{Olver1995}: given a multiplier, one considers the extended space
\(\mathcal M \times \mathbb{R}^\times\) and the extended action
\[
g : (\boldsymbol x, u) \longmapsto \bigl(g \cdot \boldsymbol x,\ \mu(g,\boldsymbol x)\,u\bigr),
\]
which is a genuine (local) group action precisely when \(\mu\) satisfies the cocycle identity~\eqref{mult}.

In what follows, we push this construction one step further: we show that after choosing a suitable
cross-section in the extended space, the {standard moving frame algorithm} can be applied without
modification to produce an explicit generating set for the field of \(\mu\)-relative invariants $\mathcal F(\mathcal M)^G_\mu$.
In particular, the relative invariants on \(\mathcal M\) are recovered from the absolute invariants of
the extended action together with a single weight generator obtained by invariantizing the added coordinate.

\medskip

Fix a smooth multiplier
\[
\mu : G \times \mathcal M \to \mathbb{R}^\times, \qquad
\mu(gh,x) = \mu\bigl(g, h \cdot x\bigr)\,\mu(h,x), \quad
\mu(e,x) = 1.
\]

Let $\boldsymbol x = (x_1, \dots, x_m)$ be a local coordinate representation of a point $\boldsymbol x \in \mathcal M$.
Consider the new manifold defined as the Cartesian product of $\mathcal M$ and the punctured real line:
\[
\widehat{\mathcal M} = \mathcal M \times \mathbb{R}^\times, \qquad \text{where } \mathbb{R}^\times = \mathbb{R} \setminus \{0\}.
\]
To the local coordinate system $(x_1, \dots, x_m)$ on $\mathcal M$, we add a new coordinate function $x_{m+1} \in \mathbb{R}^\times$.
Define the action of $G$ on a point $\boldsymbol{\hat x} \in \widehat{\mathcal M}$ by
\begin{equation}\label{twisted-xn1}
g \cdot \boldsymbol{\hat x}
\;=\;
\bigl(g \cdot \boldsymbol x,\ x_{m+1} \cdot \mu(g,\boldsymbol x)\bigr),
\quad
\boldsymbol{\hat x} = (x_1, \dots, x_m, x_{m+1}) \in \widehat{\mathcal M}.
\end{equation}

The canonical projection \(\pi : \widehat{\mathcal M} \to \mathcal M\) is given by \(\pi(\boldsymbol{\hat x}) = \boldsymbol x\).

\begin{theorem}\label{twisted-action}
The following statements hold:
\begin{enumerate}\itemsep=2pt
\item[(\textit{i})] The map \eqref{twisted-xn1} defines a smooth action of $G$ on $\widehat{\mathcal M}$.
\item[(\textit{ii})] The projection $\pi$ is a $G$-equivariant map.
\item[(\textit{iii})] If the action of $G$ on $\mathcal M$ is free and regular near a point $\boldsymbol x$, then the induced action of $G$ on $\widehat{\mathcal M}$ is also free and regular near any point in $\pi^{-1}(\boldsymbol x)$.
\end{enumerate}
\end{theorem}

\begin{proof}
(\textit{i}) For the identity $e \in G$, we compute:
$$
e \cdot \boldsymbol{\hat x}
= (e \cdot \boldsymbol x,\ x_{m+1} \mu(e,\boldsymbol x))
= (\boldsymbol x,\ x_{m+1}) = \boldsymbol{\hat x},
$$
since $\mu(e, \boldsymbol x) = 1$.

For any $g,h \in G$, using the cocycle identity for $\mu$, we have
\begin{align*}
g \cdot (h \cdot \boldsymbol{\hat x})
&= g \cdot \bigl(h \cdot \boldsymbol x,\ x_{m+1} \mu(h,\boldsymbol x)\bigr) \\
&= \bigl(gh \cdot \boldsymbol x,\ x_{m+1} \mu(h,\boldsymbol x)\, \mu(g, h \cdot \boldsymbol x)\bigr) \\
&= \bigl(gh \cdot \boldsymbol x,\ x_{m+1} \mu(gh,\boldsymbol x)\bigr)
= (gh) \cdot \boldsymbol{\hat x}.
\end{align*}
Smoothness follows from the smoothness of the action on $\mathcal M$ and of the function $\mu$.

\medskip

(\textit{ii}) To verify equivariance of $\pi$:
$$
\pi(g \cdot \boldsymbol{\hat x}) = \pi(g \cdot \boldsymbol x,\ x_{m+1} \mu(g,\boldsymbol x)) = g \cdot \boldsymbol x = g \cdot \pi(\boldsymbol{\hat x}).
$$

\medskip

(\textit{iii}) Freeness: if the stabilizer $G_{\boldsymbol x} = \{ g \in G \mid g \cdot \boldsymbol x = \boldsymbol x \}$ is trivial, then so is
\[
G_{\boldsymbol{\hat x}} = \{ g \in G \mid g \cdot \boldsymbol x = \boldsymbol x \ \text{and} \ x_{m+1} \mu(g, \boldsymbol x) = x_{m+1} \}
= \{ g \in G_{\boldsymbol x} \mid \mu(g,\boldsymbol x) = 1 \} = \{e\}.
\]

\medskip

Semi-regularity: The orbit of a point $\boldsymbol{\hat x} \in \widehat{\mathcal M}$ takes the form
\[
\mathcal O_{\widehat{\mathcal M}}(\boldsymbol{\hat x})
= \{ (g \cdot \boldsymbol x,\ x_{m+1} \mu(g,\boldsymbol x)) \mid g \in G \}.
\]
Since $\pi$ is $G$-equivariant, it maps orbits to orbits, and the restriction
\[
\pi : \mathcal O_{\widehat{\mathcal M}}(\boldsymbol{\hat x}) \to \mathcal O_{\mathcal M}(\boldsymbol x)
\]
is bijective. Indeed, for each $\boldsymbol y \in \mathcal O_{\mathcal M}(\boldsymbol x)$, there exists a unique $g_{\boldsymbol y} \in G$ such that $\boldsymbol y = g_{\boldsymbol y} \cdot \boldsymbol x$; then, using the cocycle identity for $\mu$ and the equivariance of $\pi$, the inverse map to $\pi|_{\mathcal O_{\widehat{\mathcal M}}(\boldsymbol{\hat x})}$ is given by
\[
\boldsymbol y \mapsto \bigl(\boldsymbol y,\ x_{m+1} \mu(g_{\boldsymbol y}, \boldsymbol x)\bigr).
\]
Moreover, this inverse is smooth, hence $\pi$ restricts to a diffeomorphism between orbits, and
\[
\dim \mathcal O_{\widehat{\mathcal M}}(\boldsymbol{\hat x}) = \dim \mathcal O_{\mathcal M}(\boldsymbol x),
\]
so semi-regularity is preserved.

\medskip

Regularity: Let the action of $G$ be free and regular in a neighborhood $U \subset \mathcal M$ of $\boldsymbol x$, so that the intersection $U \cap \mathcal O_{\mathcal M}(\boldsymbol x)$ is connected.
Choose $\varepsilon > 0$ such that the interval $(x_{m+1} - \varepsilon,\ x_{m+1} + \varepsilon) \subset \mathbb{R}^\times$ and set
\[
\widehat U = U \times (x_{m+1} - \varepsilon,\ x_{m+1} + \varepsilon) \subset \widehat{\mathcal M}.
\]
Then $\pi$ maps $\widehat U \cap \mathcal O_{\widehat{\mathcal M}}(\boldsymbol{\hat x})$ onto $U \cap \mathcal O_{\mathcal M}(\boldsymbol x)$.
Since the action is free, for each $\boldsymbol y \in U \cap \mathcal O_{\mathcal M}(\boldsymbol x)$ there exists a unique $g_{\boldsymbol y} \in G$ such that $\boldsymbol y = g_{\boldsymbol y} \cdot \boldsymbol x$. Due to the regularity of the orbit map $G \to \mathcal O_{\mathcal M}(\boldsymbol x)$, the map $\boldsymbol y \mapsto g_{\boldsymbol y}$ can be chosen locally smoothly. Define the map
\[
\Psi : U \cap \mathcal O_{\mathcal M}(\boldsymbol x) \longrightarrow \widehat U \cap \mathcal O_{\widehat{\mathcal M}}(\boldsymbol{\hat x}), \qquad
\Psi(\boldsymbol y) = \bigl(\boldsymbol y,\ x_{m+1} \mu(g_{\boldsymbol y}, \boldsymbol x)\bigr).
\]
This map is well-defined and smooth, and its inverse is given by $\pi$, since
\[
\pi(\Psi(\boldsymbol y)) = \boldsymbol y, \qquad \Psi(\pi(\boldsymbol{\hat y})) = \boldsymbol{\hat y}
\]
for all $\boldsymbol{\hat y} \in \widehat U \cap \mathcal O_{\widehat{\mathcal M}}(\boldsymbol{\hat x})$.
Hence $\Psi$ is a diffeomorphism.

Thus, $\widehat U \cap \mathcal O_{\widehat{\mathcal M}}(\boldsymbol{\hat x})$ is diffeomorphic to a connected subset, and the group action on $\widehat{\mathcal M}$ is regular near $\boldsymbol{\hat x}$, as claimed.
\end{proof}

\medskip
We have shown that the existence of a local moving frame on $\mathcal M$ implies its existence on the
extended manifold $\widehat{\mathcal M}$. We now show that, in fact, these moving frames correspond to
the same group element.

\begin{theorem}\label{rho}
Suppose that the action of $G$ on $\mathcal M$ is free and regular in a neighborhood of a point
$\boldsymbol x \in \mathcal M$, and let $\rho : \mathcal M \to G$ be the corresponding right moving frame.
Then, in a neighborhood of the point $\boldsymbol{\hat x} = (\boldsymbol x, x_{m+1}) \in \widehat{\mathcal M}$,
there exists a right moving frame $\widehat{\rho} : \widehat{\mathcal M} \to G$ such that
\[
\widehat{\rho}(\boldsymbol{\hat x}) = \rho(\boldsymbol x).
\]
Moreover, locally $\widehat{\rho} = \rho \circ \pi$.
\end{theorem}

\begin{proof}
Since the action of $G$ on $\mathcal M$ is free and regular in a neighborhood of $\boldsymbol x$, there exists
a local cross-section $\mathcal K \subset \mathcal M$ and a (locally) unique right moving frame
$\rho : \mathcal M \to G$, whose parameters are determined by the normalization equations
$\rho(\boldsymbol x) \cdot \boldsymbol x \in \mathcal K$ near $\boldsymbol x$.

Consider the lifted cross-section in the extended manifold,
\[
\widehat{\mathcal K} = \mathcal K \cap \{\, x_{m+1} = 1 \,\} \subset \widehat{\mathcal M}.
\]
By Theorem~\ref{twisted-action}, the action of $G$ on $\widehat{\mathcal M}$ remains free and regular,
and hence there exists a unique moving frame $\widehat{\rho}(\boldsymbol{\hat x}) \in G$ such that
\[
\widehat{\rho}(\boldsymbol{\hat x}) \cdot \boldsymbol{\hat x}
=
\Bigl(\widehat{\rho}(\boldsymbol{\hat x}) \cdot \boldsymbol x,\ 
x_{m+1}\,\mu\bigl(\widehat{\rho}(\boldsymbol{\hat x}), \boldsymbol x\bigr)\Bigr)
\in \widehat{\mathcal K}.
\]
Membership in $\widehat{\mathcal K}$ is equivalent to the system
\[
\widehat{\rho}(\boldsymbol{\hat x}) \cdot \boldsymbol x \in \mathcal K,
\qquad
x_{m+1}\,\mu\bigl(\widehat{\rho}(\boldsymbol{\hat x}), \boldsymbol x\bigr)=1.
\]
The second equation merely fixes a unique representative in the fiber with respect to the coordinate
$x_{m+1}$ and does not affect the solution of the normalization equations on $\mathcal M$.
The first equation uniquely determines the group element via the base coordinates and coincides
with the normalization equations defining $\rho(\boldsymbol x)$. By uniqueness, we obtain
\[
\widehat{\rho}(\boldsymbol{\hat x}) = \rho(\boldsymbol x).
\]

By definition of the projection $\pi$, we have
\[
\widehat{\rho}(\boldsymbol{\hat x}) = \rho(\boldsymbol x) = (\rho \circ \pi)(\boldsymbol{\hat x}).
\]

Since $\widehat{\rho} = \rho \circ \pi$ and $\pi$ is $G$-equivariant, it follows that
\[
\widehat{\rho}(g \cdot \boldsymbol{\hat x})
= \rho(\pi(g \cdot \boldsymbol{\hat x}))
= \rho(g \cdot \boldsymbol x)
= \rho(\boldsymbol x)\,g^{-1}
= \widehat{\rho}(\boldsymbol{\hat x})\,g^{-1},
\]
so right equivariance holds.

Finally, we verify that the image of a point remains in the cross-section $\widehat{\mathcal K}$.
Let $\boldsymbol{\hat x}' = g \cdot \boldsymbol{\hat x}
= (g \cdot \boldsymbol x,\ x_{m+1}\,\mu(g,\boldsymbol x))$.
Then $\widehat{\rho}(\boldsymbol{\hat x}') = \rho(g \cdot \boldsymbol x) = \rho(\boldsymbol x)\,g^{-1}$,
and by the cocycle identity for $\mu$,
\[
\mu\bigl(\rho(g \cdot \boldsymbol x),\,g \cdot \boldsymbol x\bigr)
=
\mu\bigl(\rho(\boldsymbol x)g^{-1},\,g \cdot \boldsymbol x\bigr)
=
\frac{\mu(\rho(\boldsymbol x),\boldsymbol x)}{\mu(g,\boldsymbol x)}.
\]
Therefore,
\[
x_{m+1}'\,\mu\bigl(\widehat{\rho}(\boldsymbol{\hat x}'),\pi(\boldsymbol{\hat x}')\bigr)
=
x_{m+1}\,\mu(\rho(\boldsymbol x),\boldsymbol x)
=1,
\]
so the gauge condition is preserved.
\end{proof}

\medskip

We now prove that invariantization of the multiplier produces a relative invariant of weight $-1$, and that its repeated invariantization yields the trivial absolute invariant.

\begin{theorem}\label{inv-jak}
Let $G$ act smoothly on $\mathcal M$, and let
$\mu : G \times \mathcal M \to \mathbb{R}^\times$ be a multiplier.
Suppose that a right moving frame $\rho( \boldsymbol x) \in G$ exists in a neighborhood of $\boldsymbol x \in \mathcal M$.
Then:

\begin{enumerate}
\item[$(i)$] The {invariantized multiplier}
\[
\mu\bigl(\rho(\boldsymbol x),\boldsymbol x\bigr)
\]
is a relative invariant of weight $-1$ with multiplier $\mu$, i.e.,
\[
\mu\bigl(\rho(g \cdot \boldsymbol x),\,g \cdot \boldsymbol x\bigr)
=
\mu(g,\boldsymbol x)^{-1}\,\mu\bigl(\rho(\boldsymbol x),\boldsymbol x\bigr),
\qquad \forall\, g\in G.
\]

\item[$(ii)$] The repeated invariantization of the multiplier equals $1$, namely
\[
\iota\!\Bigl(\mu\bigl(\rho(\boldsymbol x),\boldsymbol x\bigr)\Bigr)=1.
\]
\end{enumerate}
\end{theorem}

\begin{proof}
$(i)$ Fix $g \in G$ and $\boldsymbol x \in \mathcal M$. Using equivariance of the moving frame,
\[
\rho(g \cdot \boldsymbol x) = \rho(\boldsymbol x)\,g^{-1},
\]
and substituting $h = g^{-1}$ and $g = \rho(\boldsymbol x)$ into the cocycle identity~\eqref{mult},
we obtain
\[
\mu\bigl(\rho(\boldsymbol x)g^{-1},\,g \cdot \boldsymbol x\bigr)
=
\frac{\mu(\rho(\boldsymbol x),\boldsymbol x)}{\mu(g,\boldsymbol x)}.
\]
Hence
\[
\mu\bigl(\rho(g \cdot \boldsymbol x),\,g \cdot \boldsymbol x\bigr)
=
\mu(g,\boldsymbol x)^{-1}\,\mu\bigl(\rho(\boldsymbol x),\boldsymbol x\bigr),
\]
which proves the claim.

$(ii)$  By definition of invariantization,
\[
\iota\!\Bigl(\mu\bigl(\rho(\boldsymbol x),\boldsymbol x\bigr)\Bigr)
=
\mu\Bigl(\rho(\rho(\boldsymbol x)\!\cdot\!\boldsymbol x),\,\rho(\boldsymbol x)\!\cdot\!\boldsymbol x\Bigr).
\]
Applying equivariance with $g=\rho(\boldsymbol x)$ gives
\[
\rho(\rho(\boldsymbol x)\!\cdot\!\boldsymbol x)=\rho(\boldsymbol x)\,\rho(\boldsymbol x)^{-1}=e,
\]
and therefore
\[
\iota\!\Bigl(\mu\bigl(\rho(\boldsymbol x),\boldsymbol x\bigr)\Bigr)
=
\mu\bigl(e,\,\rho(\boldsymbol x)\!\cdot\!\boldsymbol x\bigr)=1,
\]
since $\mu(e,\cdot)=1$.
\end{proof}

\medskip

We now establish the relationship between absolute and relative invariants.

\begin{theorem}\label{gt}
Let $G$ act smoothly on $\mathcal M$ with multiplier $\mu : G \times \mathcal M \to \mathbb{R}^\times$,
and let it act in the twisted manner on $\widehat{\mathcal M} = \mathcal M \times \mathbb{R}^\times$.
Assume that a right moving frame $\rho( \boldsymbol x) \in G$ exists in a neighborhood of a generic point.
Then:
\begin{enumerate}
\item[(i)] The invariant fields coincide:
\[
\mathcal F(\mathcal M)^G_\mu = \mathcal F(\widehat{\mathcal M})^G.
\]
\item[(ii)] If $\mathcal F(\mathcal M)^G$ is generated by fundamental invariants
$I_1(\boldsymbol x),\dots,I_{m-r}(\boldsymbol x)$, then
$\mathcal F(\mathcal M)^G_\mu$ is generated by
$I_1(\boldsymbol x),\dots,I_{m-r}(\boldsymbol x)$ together with the single relative invariant
$\mu(\rho(\boldsymbol x),\boldsymbol x)$ of weight $-1$.
In particular, any relative invariant of weight $\omega$ has the form
\[
A(\boldsymbol x)
=
\mu(\rho(\boldsymbol x),\boldsymbol x)^{-\omega}\,
F\bigl(I_1(\boldsymbol x),\dots,I_{m-r}(\boldsymbol x)\bigr),
\]
for some rational function $F$.
\end{enumerate}
\end{theorem}

\begin{proof}

By Theorem~\ref{rho} and the choice of cross-section $\widehat{\mathcal K}$,
the invariantization of the coordinates of a point $\boldsymbol{\hat x}$ has the form
$$
\hat{\iota}(\boldsymbol{\hat x}) = \widehat{\rho}(\boldsymbol{\hat x}) \cdot \boldsymbol{\hat x}
= (\iota(\boldsymbol x),\ \widehat{\iota}(x_{m+1}))
= \left(\iota(\boldsymbol x),\ x_{m+1} \cdot \mu(\rho(\boldsymbol x), \boldsymbol x)\right).
$$
As shown above, the invariantized multiplier $\mu(\rho(\boldsymbol x), \boldsymbol x)$
is a relative invariant of weight $-1$.

\medskip

To prove the inclusion $\mathcal F(\mathcal M)^G_\mu \subseteq \mathcal F(\widehat{\mathcal M})^G$,
let $A \in \mathcal F(\mathcal M)^G_\mu$ be a relative invariant of weight $\omega$:
$$
A(g \cdot \boldsymbol x) = \mu(g, \boldsymbol x)^\omega\, A(\boldsymbol x).
$$
Define a function on $\widehat{\mathcal M}$ by
\[
\widehat{A}(\boldsymbol x, x_{m+1}) = x_{m+1}^{-\omega} A(\boldsymbol x).
\]
Then
\[
\widehat{A}(g \cdot \boldsymbol x,\ x_{m+1} \mu(g, \boldsymbol x))
= (x_{m+1} \mu(g, \boldsymbol x))^{-\omega} A(g \cdot \boldsymbol x)
= x_{m+1}^{-\omega} \mu(g, \boldsymbol x)^{-\omega} \mu(g, \boldsymbol x)^{\omega} A(\boldsymbol x)
= \widehat{A}(\boldsymbol x, x_{m+1}),
\]
so $\widehat{A} \in \mathcal F(\widehat{\mathcal M})^G$.

\medskip

To prove the inclusion $\mathcal F(\widehat{\mathcal M})^G \subseteq \mathcal F(\mathcal M)^G_\mu$,
let $F \in \mathcal F(\widehat{\mathcal M})^G$. Since invariantization of coordinates generates all invariants,
$F$ can be written as a function of the invariantized coordinates:
$$
F(\boldsymbol x, x_{m+1}) = \Phi\bigl(\hat{\iota}(\boldsymbol{\hat x})\bigr)
= \Phi\bigl(\iota(\boldsymbol x),\ \mu(\rho(\boldsymbol x), \boldsymbol x)\bigr),
$$
for some function $\Phi$. Since $\rho(\boldsymbol x) \cdot \boldsymbol x$ in normalized coordinates depends
only on the fundamental absolute invariants $I_1(\boldsymbol x), \dots, I_{m-r}(\boldsymbol x)$,
there exists a function $\widetilde{\Phi}$ such that
$$
F(\boldsymbol x, x_{m+1}) = \widetilde{\Phi}\bigl(I_1(\boldsymbol x),\dots,I_{m-r}(\boldsymbol x),\ \mu(\rho(\boldsymbol x), \boldsymbol x)\bigr).
$$
Since the $I_j(\boldsymbol x)$ are absolute invariants and $\mu(\rho(\boldsymbol x), \boldsymbol x)$ is a relative invariant,
it follows that $F \in \mathcal F(\mathcal M)^G_\mu$.

From the two inclusions, we conclude that the invariant fields coincide:
$$
\mathcal F(\mathcal M)^G_\mu = \mathcal F(\widehat{\mathcal M})^G,
$$
which proves part (i).

\medskip

From the previous steps, it follows that any invariant on $\widehat{\mathcal M}$—and hence any relative
invariant on $\mathcal M$—is a rational function of the fundamental absolute invariants
$I_1(\boldsymbol x), \dots, I_{m-r}(\boldsymbol x)$ and the relative invariant
$\mu(\rho(\boldsymbol x), \boldsymbol x)$.
Let $A$ be a relative invariant of weight $\omega$. Then the product
$A \cdot \mu(\rho(\boldsymbol x), \boldsymbol x)^{\omega}$ is an absolute invariant,
say $F(I_1(\boldsymbol x), \dots, I_{m-r}(\boldsymbol x))$ for some function $F$ in $m - r$ variables.
Thus,
$$
A = \mu(\rho(\boldsymbol x), \boldsymbol x)^{-\omega}
\, F(I_1(\boldsymbol x), \dots, I_{m-r}(\boldsymbol x)),
$$
which proves part (ii).
\end{proof}

\medskip

The proposed construction does not alter the standard moving frame algorithm:
the new approach is fully compatible with the classical theory while providing an
{algorithmic} procedure for constructing relative invariants entirely within the
original moving frame framework.

\medskip

A general description of the structure of scalar relative invariants is already implicit in the
classical theory. In particular, it is shown in \cite{Olver1995} (where the scalar multiplier is
referred to as a {weight function}) that once a nonvanishing relative invariant \(R_0\) of a given
weight is fixed, any other relative invariant of the same weight is obtained by multiplying \(R_0\)
by an absolute invariant. Theorem~\ref{gt} refines this statement in our setting by providing a
canonical choice of such a distinguished generator: namely, the invariantized multiplier
\(\mu(\rho(\boldsymbol x),\boldsymbol x)\) produced by the moving frame.

\medskip

Interestingly, any nonvanishing function $F \in \mathcal F(\mathcal M)$ is a relative invariant
for the group action with multiplier $d^0 F$.
In particular, this observation provides a direct link to the realizability problem for
\(1\)-cocycles. Namely, a cocycle \(\mu\in Z^{1}(G,\mathcal F(\mathcal M)^\times)\) is said to be realizable
(as a multiplier) if there exists a nonvanishing function \(F\in\mathcal F(\mathcal M)^\times\) such that
\(\mu=d^{0}F\).
Our remark shows that every such \(F\) automatically yields a relative invariant 
with multiplier \(d^{0}F\), while conversely, \(\mu\) is realizable precisely when it arises in this way.
Therefore, in the setting considered in this paper (regular actions admitting a moving frame and the
rational function class), the triviality of the first cohomology implies that {every} multiplier
\(1\)-cocycle is realizable, and hence can be incorporated into the moving-frame construction of relative
invariants.

\section{Relative Invariants of the Plane Projective Group and Integral Invariants}

Let us consider a practical application of the algorithm for constructing relative invariants in image analysis tasks.

Let \(F \in \mathcal{F}(\mathcal{M})\). Suppose that for some neighborhood \(U \subset \mathcal{M}\), the integral
$$
\int_U F(\boldsymbol{x})\,d\boldsymbol{x}, \quad \text{where} \quad d\boldsymbol{x} = dx_1 \cdots dx_n,
$$
is defined. We call this integral an {integral \(G\)-invariant} if for every \(g \in G\), its value remains unchanged under the change of variables \(\boldsymbol{x} \mapsto g \cdot \boldsymbol{x}\), i.e.,
$$
\int_{g \cdot U} F(\boldsymbol{g \cdot \boldsymbol{x}})\,d\boldsymbol{x}
\;=\;
\int_U F(\boldsymbol{x})\,d\boldsymbol{x}.
$$
Applying the change of variables formula to the left-hand side yields the condition:
$$
\int_{U} F(g \cdot \boldsymbol{x})\, J(g, \boldsymbol{x})\,d\boldsymbol{x}
\;=\;
\int_U F(\boldsymbol{x})\,d\boldsymbol{x},
$$
where \(J(g, \boldsymbol{x}) = \det \frac{\partial (g \cdot \boldsymbol{x})}{\partial \boldsymbol{x}}\) is the Jacobian of the transformation.

Thus, a sufficient condition for integral invariance is the invariance of the integrand \(F(\boldsymbol{x})\). This is equivalent to the requirement:
$$
F(g \cdot \boldsymbol{x})\,J(g, \boldsymbol{x}) = F(\boldsymbol{x})
\quad \Longleftrightarrow \quad
F(g \cdot \boldsymbol{x}) = J(g, \boldsymbol{x})^{-1} F(\boldsymbol{x}).
$$
Therefore, integral invariants naturally arise when the integrand \(F(\boldsymbol{x})\) is a relative invariant of weight \(\omega=-1\) with respect to the Jacobian multiplier \(J(g, \boldsymbol{x})\).

The proposed construction is deeply rooted in the theory of invariant Haar measures on homogeneous spaces, as the condition $F(g \cdot \boldsymbol{x}) = |J(g, \boldsymbol{x})|^{-1} F(\boldsymbol{x})$ effectively defines a density $F(\boldsymbol{x})$ that transforms the standard Lebesgue measure $d\boldsymbol{x}$ into a $G$-invariant measure $d\lambda = F(\boldsymbol{x}) d\boldsymbol{x}$ on the manifold $\mathcal{M}$. This framework ensures the consistent definition of geometric features that remain independent of the coordinate system or camera viewpoint.

\medskip

Note that integral invariants play a key role in image analysis. A grayscale image is mathematically modeled as a piecewise-continuous real-valued positive function \(u(x,y)\) defined on a compact domain \(\Omega \subset \mathbb{R}^2\), such that
$$
 \int_{\Omega} u(x,y)\,dx\,dy < \infty.
$$
The value of the function \(u(x,y)\) is interpreted as the brightness (intensity) of the pixel at the corresponding coordinates.

Integral, and more generally, moment invariants of the affine group of the plane and its subgroups are global features of the image and are widely used in pattern recognition tasks (see~\cite{FSB}). 
However, the situation becomes significantly more complicated when transitioning from the affine to the projective group of the plane. All known attempts to extend the affine theory by defining projective analogs of moments and constructing invariant expressions based on them have so far been unsuccessful. In fact, the problem of complete characterization of projective invariants of images under general projective transformations remains open (see~\cite{Open}).

In the next subsection, we provide a complete description of the field of relative joint projective invariants and, using them, construct a new family of joint projective invariants of images.

\subsection{Image Joint Projective Invariants}

Let \(G = PGL(3,\mathbb{R})\) act on \(\mathbb{R}^2\) via projective transformations
\begin{equation}\label{pgl-action-point}
g:(x,y)\ \longmapsto\
\left(
  \frac{a_1x+a_2y+a_3}{c_1x+c_2y+c_3},\;
  \frac{b_1x+b_2y+b_3}{c_1x+c_2y+c_3}
\right),
\qquad
g=\begin{pmatrix}
a_1&a_2&a_3\\
b_1&b_2&b_3\\
c_1&c_2&c_3
\end{pmatrix}\in GL(3,\mathbb{R}),
\end{equation}
where matrices are considered up to a nonzero scalar multiple.

Consider the induced diagonal action of \(G\) on the configuration space
\(\mathcal M = (\mathbb{R}^2)^n\):
$$
g:\ (x_i,y_i)\ \longmapsto\ g\cdot(x_i,y_i),\qquad i=1,\dots,n.
$$

For a single point \((x_i,y_i)\), the Jacobian of the planar projective transformation \eqref{pgl-action-point} is
$$
J_i(g,(x_i,y_i))
\;=\;
\frac{\det(g)}{\bigl(c_1x_i+c_2y_i+c_3\bigr)^3},
\qquad
\det(g)=
\begin{vmatrix}
a_1 & a_2 & a_3 \\
b_1 & b_2 & b_3 \\
c_1 & c_2 & c_3
\end{vmatrix}.
$$
Hence, for the diagonal action on \(\mathcal M\), the \emph{total Jacobian multiplier} is the product
$$
J(g,\bs x)\;=\;\prod_{i=1}^n J_i\bigl(g,(x_i,y_i)\bigr),
\qquad
\bs x=(x_1,y_1,\dots,x_n,y_n)\in\mathcal M.
$$

We then construct a moving frame \(\rho(\bs x)\) for the projective action on \(\mathcal M\)
(on an open subset of general position), and then invariantize the Jacobian multiplier \(J(g,\bs x)\). This yields a relative invariant of weight \(-1\), namely the invariantized Jacobian
\(\iota(J)(g,\bs x)=J(\rho(\bs x),\bs x)\), which can be combined with absolute invariants
to form projective integral invariants of the image data.

\subsection{Construction of the Moving Frame}

The projective group of the plane, $PGL(3, \mathbb{R})$, is an 8-parameter Lie group. Since each point in the projective plane provides two coordinates, a configuration of $n$ points in general position forms a $2n$-dimensional manifold. To ensure that the group action is locally free, the dimension of this configuration space must satisfy $2n \geq 8$. Thus, we consider the case $n=4$, which provides a sufficient number of equations to uniquely determine all group parameters.

We fix the value of the parameter \(c_3 = 1\) and choose the following cross-section of the manifold \(\mathcal{M}\):

$$
\mathcal{K} = \{ x_1 = 0,\; y_1 = -1,\; x_2 = 1,\; y_2 = 1,\; x_3 = 1,\; y_3 = 0,\; x_4 = 0,\; y_4 = 0 \}
$$

Solving the normalization equations yields the following values of the moving frame parameters:

\begin{gather*}
a_{{1}}=\frac {(y_1-y_4)\delta_{123} \delta_{234}}{\Delta}, \quad
a_{{2}}=-\frac {(x_1-x_{{4}}) \delta_{234} \delta_{123}}{\Delta}, \quad
a_{{3}}=\frac {\delta_{123} \delta_{234} \begin{vmatrix} x_1 & y_1 \\ x_4 & y_4\end{vmatrix}}{\Delta},
\\
b_{{1}}=-\frac { (y_3-y_4)\delta_{123} \delta_{124}}{\Delta}, \quad 
b_{{2}}={\frac {(x_3-x_4)\delta_{123} \delta_{124}}{\Delta}}, \quad
b_{{3}}=-{\frac { \delta_{123} \delta_{124} \begin{vmatrix} x_3 & y_3 \\ x_4 & y_4\end{vmatrix} }{\Delta}},
\\
c_1=\frac{(y_1-y_4) \delta_{123} \delta_{234}+(y_2-y_3) \delta_{124} \delta_{134}}{\Delta
}, \quad
c_2=-\frac{(x_1-x_4) \delta_{123} \delta_{234}+(x_2-x_3) \delta_{124} \delta_{134}}{\Delta
}.
\end{gather*}
where   
$$ \delta_{ijk}=\begin{vmatrix} x_i & x_j & x_k \\ y_i & y_j & y_k \\ 1 & 1 & 1 \end{vmatrix},  \quad \Delta= \delta_{123} \delta_{234} \begin{vmatrix} x_1 & y_1 \\ x_4 & y_4\end{vmatrix}+\delta_{134} \delta_{124} \begin{vmatrix} x_2 & y_2 \\ x_2 & y_2\end{vmatrix}.$$

We now invariantize the coordinate functions \(x_i, y_i\) for \(i = 5, \ldots, n\), which were not involved in the normalization. This yields
\begin{gather*}
\iota(x_i) = \frac{\delta_{234} \delta_{123} \delta_{14i}}{\delta_{234} \delta_{123} \delta_{14i} + \delta_{134} \delta_{124} \delta_{23i}}, \\
\iota(y_i) = -\frac{\delta_{124} \delta_{123} \delta_{34i}}{\delta_{234} \delta_{123} \delta_{14i} + \delta_{134} \delta_{124} \delta_{23i}}.
\end{gather*}

These functions generate the field of absolute invariants. However, we attempt to find a simpler set of generators. We compute:
\begin{gather*}
\iota(x_i)^{-1} = \frac{\delta_{234} \delta_{123} \delta_{14i} + \delta_{134} \delta_{124} \delta_{23i}}{\delta_{234} \delta_{123} \delta_{14i}} = 1 + \frac{\delta_{134} \delta_{124} \delta_{23i}}{\delta_{234} \delta_{123} \delta_{14i}}, \\
\frac{\iota(x_i)}{\iota(y_i)} = -\frac{\delta_{234} \delta_{14i}}{\delta_{124} \delta_{34i}}.
\end{gather*}

Thus, the field of joint absolute invariants of the planar projective group is generated by the following \(2(n-4)\) invariants:
$$
I^{(1)}_i(\boldsymbol{x}) = \frac{\delta_{134} \delta_{124} \delta_{23i}}{\delta_{234} \delta_{123} \delta_{14i}}, \qquad
I^{(2)}_i(\boldsymbol{x}) = \frac{\delta_{234} \delta_{14i}}{\delta_{124} \delta_{34i}}, \qquad i = 5, \ldots, n.
$$
For \(n = 4\), there are no absolute invariants.

Now consider the multiplier, which is the Jacobian of the projective transformation:
$$
\mu(g, \boldsymbol{x}) = J_n(g, \boldsymbol{x}) = \frac{\det(g)^n}{\prod_{i=1}^n s_i^3} = \prod_{i=1}^n \frac{\det(g)}{s_i^3}, \quad s_i = c_1 x_i + c_2 y_i + 1.
$$

According to Theorem~\ref{inv-jak}, its invariantization \(J(g \cdot \boldsymbol{x}, \boldsymbol{x})\) is a relative invariant of weight \(-1\). Substituting the moving frame parameters into the expressions for the multipliers yields the invariantized Jacobian:

\begin{gather*}
\iota\left(\frac{\det(g)}{s_1^3} \right)=-\frac{\delta_{234}^2}{\delta_{123} \delta_{124} \delta_{134}}, \iota\left(\frac{\det(g)}{s_2^3} \right)=\frac{\delta_{134}^2}{\delta_{123} \delta_{124} \delta_{234}}, \\
\iota\left(\frac{\det(g)}{s_3^3} \right)=\frac{\delta_{124}^2}{\delta_{123} \delta_{234} \delta_{134}}, \iota\left(\frac{\det(g)}{s_4^3} \right)=-\frac{\delta_{123}^2}{\delta_{124} \delta_{134} \delta_{234}},
\end{gather*}
and for  $i>4$:
$$
\iota\left(\frac{\det(g)}{s_i^3} \right)=-\frac{(\delta_{123} \delta_{124} \delta_{134} \delta_{234})^{2}}{(\delta_{123}  \delta_{234} \delta_{14i}+\delta_{124} \delta_{134}  \delta_{23i})^3}.
$$

Thus, the invariantized Jacobian (after  simplifications and ignoring the overall sign) takes the form
\begin{equation}\label{n>4}
J\bigl(g\!\cdot\!\bs x,\bs x\bigr)
=
\bigl(\delta_{123}\delta_{124}\delta_{134}\delta_{234}\bigr)^{2n-1}
\prod_{i=5}^{n}
\frac{1}{\bigl(\delta_{123}\delta_{234}\delta_{14i}+\delta_{124}\delta_{134}\delta_{23i}\bigr)^3},
\qquad n>4,
\end{equation}
and for  $n=4$:
\begin{equation*}\label{n=4}
J\bigl(g\!\cdot\!\bs x,\bs x\bigr)
=
\bigl(\delta_{123}\delta_{124}\delta_{134}\delta_{234}\bigr)^{-1}.
\end{equation*}

\medskip

Although the relative invariance of the invariantized Jacobian is guaranteed by the very existence of a moving frame,
it is instructive to see explicitly the mechanism by which the Jacobian multiplier arises.
Technically, everything reduces to the fact that each determinant $\delta_{ijk}$ is itself a relative invariant: under the action of
$g \in PGL(3,\mathbb{R})$ it transforms according to the rule
\begin{equation}\label{delta-transform}
\delta_{ijk}\ \longmapsto\ \frac{\det(g)}{s_i s_j s_k}\,\delta_{ijk}.
\end{equation}

It follows that in the explicit formulas for the fundamental absolute invariants
$I^{(1)}_i(\boldsymbol{x})$ and $I^{(2)}_i(\boldsymbol{x})$
all such factors cancel pairwise, which explains their invariance.
Of course, these invariants could also be obtained without the moving frame method;
however, the moving frame approach immediately guarantees their functional completeness
and the absence of additional independent generators.

For relative invariants the situation is analogous, but now one must track the total multiplier
that reproduces the Jacobian to the power $-1$.
The key point here is the transformation of the mixed combination
\begin{equation}\label{mixed-transform}
\delta_{123}\delta_{234}\delta_{14i}+\delta_{124}\delta_{134}\delta_{23i}
\ \longmapsto\
\frac{\det(g)^3}{s_1^2 s_2^2 s_3^2 s_4^2 s_i}\,
\bigl(\delta_{123}\delta_{234}\delta_{14i}+\delta_{124}\delta_{134}\delta_{23i}\bigr),
\end{equation}
which shows that this sum behaves as a relative invariant with a well-defined multiplier.
It is precisely due to \eqref{delta-transform}--\eqref{mixed-transform} that all factors arising
from the group action on the right-hand side of \eqref{n>4} ultimately combine into the overall
Jacobian multiplier raised to the power $-1$.

\medskip

Hence, the field of joint relative projective invariants with the Jacobian multiplier is generated
by $2n-8$ absolute invariants together with the invariantized Jacobian:
$$
\mathcal F(\mathcal M)^{PGL(3,\mathbb{R})}_{J(g,\boldsymbol{x})}
=\mathbb{R}\bigl(
I^{(1)}_5(\boldsymbol{x}),\ I^{(2)}_5(\boldsymbol{x}),\ldots,
I^{(1)}_n(\boldsymbol{x}),\ I^{(2)}_n(\boldsymbol{x}),
J(g\!\cdot\!\boldsymbol{x},\boldsymbol{x})
\bigr),
\qquad \dim \mathcal M = 2n \ge 8.
$$
For $n=4$ there are no absolute invariants, and the field of relative invariants is generated solely
by the invariantized Jacobian.

It is interesting to note that for $n=3$ the diagonal projective action of $PGL(3,\mathbb{R})$ on
$(\mathbb{R}^2)^3$ is not locally free, since
$$
\dim(\mathbb{R}^2)^3 = 6 < \dim PGL(3,\mathbb{R}) = 8,
$$
and therefore a moving frame in the standard sense does not exist.
Nevertheless, the determinant
$$
\delta_{123} =
\begin{vmatrix}
x_1 & x_2 & x_3 \\
y_1 & y_2 & y_3 \\
1   & 1   & 1
\end{vmatrix}
$$
is a relative invariant of weight $-\tfrac{1}{3}$.
This example highlights the fact that relative invariants may exist even in situations
where the lack of local freeness or regularity prevents the existence of a moving frame.

\subsection{Image projective invariants}

The obtained description of the field of relative invariants with the Jacobian multiplier allows us to derive a complete description of 
{integral projective invariants of images}.

Let $u(x,y)$ be an image function defined on a compact domain in $\mathbb{R}^2$.
Consider $n$ copies of this image,
\[
u_i(x_i,y_i) = u(x_i,y_i),\qquad i=1,\dots,n.
\]
We assume a {trivial} action of the projective group on these copies:
$$
g\cdot u_i \;=\; u_i,
$$
This is merely a formalization of the fact that the pixel intensity values of the image remain unchanged under projective deformations of the domain. Hence,
each \(u_i\) is an absolute invariant under the induced action on the extended space
\((\mathbb{R}^2\times\mathbb{R})^n\).

For $\bs \alpha =(\alpha_1, \ldots, \alpha_n),\ \bs \beta =(\beta_1, \ldots, \beta_n)\in \mathbb{Z}^{n},\ n>4$, define the following family of integral invariants:
\begin{equation*}
\mathcal I_{\bs \alpha, \bs \beta}\ =\ 
\int_{\mathbb{R}^{2n}}
J\bigl(g\!\cdot\!\bs x,\bs x\bigr)\,
\prod_{i=1}^{n} \left( I_i^{(1)}(\bs x)\right)^{\alpha_i}\, \left(I_i^{(2)}(\bs x)\right)^{\beta_i} u(x_i,y_i)\,
dx_1\,dy_1\cdots dx_n\,dy_n.
\end{equation*}
To unify notation, we set $ I_i^{(1)}(\bs x) = I_i^{(2)}(\bs x) = 1$ for $i \leq 4.$

All functions $u_i$ are absolute invariants, so we may include them in the integrand. This ensures that the integral invariant depends on the image and serves as a global feature that aggregates information from the entire image.

\medskip

The constructed invariants are primarily of theoretical interest: these are high-dimensional integrals, and their direct computation on real images is computationally expensive without the use of heuristic techniques. 
Integral projective joint invariants of this type were first considered in~\cite{Suk2000}, where relative invariants were obtained for the cases $n=3,4$, and a general construction scheme was proposed for arbitrary $n$, although the question of describing all relative invariants was not addressed.

\section{Conclusions}

This work presents a constructive modification of the method of moving frames, adapted for the construction of relative invariants under smooth Lie group actions. The theoretical framework justifies the transition from the problem of finding relative invariants on the base manifold $\mathcal{M}$ to the problem of constructing absolute invariants on the extended space $\widehat{\mathcal{M}} = \mathcal{M} \times \mathbb{R}^\times$. It is proven that the twisted group action on $\widehat{\mathcal{M}}$, defined via a $1$-cocycle (multiplier), preserves the freeness and regularity of the original action.

A fundamental connection is established between the classical moving frame $\rho$ on $\mathcal{M}$ and the structure of relative invariants. It is shown that the invariantization of the multiplier $J(\boldsymbol{x}) = \mu(\rho(\boldsymbol{x}), \boldsymbol{x})$ serves as a canonical generator of weight $-1$ relativity. This enables a universal normal form for any relative invariant, expressed as a product of this generator raised to the appropriate power and a function of the fundamental absolute invariants.

The proposed approach is applied to the relevant problem of constructing projective invariants of images. For the diagonal action of the projective plane group $PGL(3, \mathbb{R})$, an explicit form of the invariantized Jacobian and fundamental joint invariants is derived. Based on this, a new family of integral projective invariants (image joint projective invariants) is constructed, capable of aggregating image intensity information in a way that is invariant under projective deformations of the domain. This construction provides a theoretical foundation for the development of new image recognition algorithms robust to perspective distortions.

Future research directions include extending this methodology to higher-order projective groups, as well as developing efficient algorithms for computing the resulting invariants.


\begin{thebibliography}{30}

\bibitem{Olver1997}
Fels, M., and Olver, P. J. (1997). On relative invariants. \textit{Mathematische Annalen, 308}(4), 701--730.

\bibitem{Bott}
Bott, R. (2010). Some aspects of invariant theory in differential geometry. In E. Vesentini (Ed.), \textit{Differential operators on manifolds} (Vol. 70, pp. 49--145). Springer.

\bibitem{Graustein}
Graustein, W. C. (1930). Invariant methods in classical differential geometry. \textit{Bulletin of the American Mathematical Society, 36}(8), 489--521.

\bibitem{Olver2007}
Olver, P. J. (2007). Generating differential invariants. \textit{Journal of Mathematical Analysis and Applications, 333}(1), 450--471.

\bibitem{Helgason}
Helgason, S. (2022). \textit{Groups and geometric analysis: Integral geometry, invariant differential operators, and spherical functions} (Vol. 83). American Mathematical Society.

\bibitem{KobayashiNomizu}
Kobayashi, S., and Nomizu, K. (1963--1969). \textit{Foundations of differential geometry} (Vols. 1--2). Wiley.

\bibitem{Miyake}
Miyake, T. (2006). \textit{Modular forms}. Springer Science and Business Media.

\bibitem{Langlands}
Jacquet, H., and Langlands, R. P. (2006). \textit{Automorphic forms on GL(2): Part 1} (Vol. 114). Springer.

\bibitem{AbrahamMarsden}
Abraham, R., and Marsden, J. E. (1978). \textit{Foundations of mechanics} (2nd ed.). Benjamin/Cummings Publishing Company.

\bibitem{Arnold}
Arnold, V. I. (1989). \textit{Mathematical methods of classical mechanics} (2nd ed.). Springer-Verlag.

\bibitem{Olver2023}
Olver, P. J. (2023). Projective invariants of images. \textit{European Journal of Applied Mathematics, 34}(5), 936--946.

\bibitem{Wang}
Wang, Y. B., Wang, X. W., Zhang, B., and Wang, Y. (2015). Projective invariants of D-moments of 2D grayscale images. \textit{Journal of Mathematical Imaging and Vision, 51}(2), 248--259.

\bibitem{Open}
Li, E., Mo, H., Xu, D., and Li, H. (2019). Image projective invariants. \textit{IEEE Transactions on Pattern Analysis and Machine Intelligence, 41}(5), 1144--1157.

\bibitem{Miller}
Miller, W., Jr. (1968). \textit{Lie theory and special functions}. Academic Press.

\bibitem{VilenkinKlimyk}
Vilenkin, N. J., and Klimyk, A. U. (1991--1993). \textit{Representation of Lie groups and special functions} (Vols. 1--3). Kluwer Academic Publishers.

\bibitem{Weyl}
Weyl, H. (1950). \textit{The theory of groups and quantum mechanics}. Courier Corporation.

\bibitem{Olver1995}
Olver, P. J. (1995). \textit{Equivalence, invariants and symmetry}. Cambridge University Press.

\bibitem{Olver1998}
Fels, M., and Olver, P. J. (1998). Moving coframes: I. A practical algorithm. \textit{Acta Applicandae Mathematicae, 51}(2), 161--213.

\bibitem{Olver1999-1}
Fels, M., and Olver, P. J. (1999). Moving coframes: II. Regularization and theoretical foundations. \textit{Acta Applicandae Mathematicae, 55}(2), 127--208.

\bibitem{B-2025}
Bedratyuk, L. (2025). Joint projective invariants on first jet spaces of point configurations via moving frames. \textit{arXiv preprint arXiv:2511.18575}.

\bibitem{B-2025-1}
Bedratyuk, L. (2025). First order joint differential projective invariants. \textit{Communications in Algebra}, 1--20. https://doi.org/10.1080/00927872.2025.2569449

\bibitem{Krug}
Kruglikov, B., and Schneider, E. (2025). Invariant divisors and equivariant line bundles. \textit{Forum of Mathematics, Sigma}, \textit{13}, e68.

\bibitem{Olver2001}
Olver, P. J. (2001). Joint invariant signatures. \textit{Foundations of Computational Mathematics, 1}(1), 3--68.

\bibitem{Olver1999-2}
Olver, P. J. (1999). \textit{Classical invariant theory} (Vol. 44). Cambridge University Press.

\bibitem{Braun}
Brown, K. S. (1982). \textit{Cohomology of groups} (Vol. 87). Springer-Verlag.

\bibitem{FSB}
Flusser, J., Suk, T., and Zitova, B. (2017). \textit{2D and 3D image analysis by moments}. John Wiley and Sons.

\bibitem{Suk2000}
Suk, T., and Flusser, J. (2000). Point-based projective invariants. \textit{Pattern Recognition, 33}(2), 251--261.

\end{thebibliography}
\end{document}